\documentclass[11pt]{article}
\usepackage{latexsym,amsmath,amssymb,amsfonts,mathrsfs,amsthm}
\usepackage{epsf,graphicx,epsfig,color,cite,cases}
\usepackage{subfigure,graphics,multirow,marginnote,enumerate,bm}
\usepackage[T1]{fontenc}
\usepackage{algorithm}

\newcommand{\R}{{\mat R}}

\newcommand{\ds}{\displaystyle}

\newcommand{\be}{\begin{eqnarray}}
\newcommand{\ben}{\begin{eqnarray*}}
\newcommand{\en}{\end{eqnarray}}
\newcommand{\enn}{\end{eqnarray*}}

\newcommand{\mat}{\mathbb}

\newtheorem{theorem}{Theorem}[section]
\newtheorem{lemma}[theorem]{Lemma}

\begin{document}
\renewcommand{\theequation}{\arabic{section}.\arabic{equation}}

\title{\bf Simultaneous recovery of a locally rough interface and the embedded obstacle with the reverse time migration}

\author{
Jianliang Li\thanks{School of Mathematics and Statistics, Hunan Normal University, Changsha 410081, China ({\tt lijl@amss.ac.cn})}
\and
Jiaqing Yang\thanks{School of Mathematics and Statistics, Xi'an Jiaotong University,
Xi'an, Shaanxi 710049, China ({\tt jiaq.yang@mail.xjtu.edu.cn})}
}
\date{}

\maketitle

\begin{abstract}
Consider the inverse acoustic scattering of time-harmonic point sources by an unbounded locally rough interface 
with bounded obstacles embedded in the lower half-space. A novel version of reverse time migration is proposed 
to reconstruct both the locally rough interface and the embedded obstacle. By a modified Helmholtz-Kirchhoff identity
associated with a planar interface, we obtain a modified imaging functional which has been shown that it always 
peaks on the local perturbation of the interface and on the embedded obstacle. Numerical examples are presented 
to demonstrate the effectiveness of the method.
\vspace{.2in}


\end{abstract}

\maketitle

\section{Introduction}
This paper concerns the two-dimensional acoustic scattering of time-harmonic point sources by a locally rough interface with 
an embedded obstacle in the lower half-space. Given the incident wave, the direct scattering problem is to determine the distribution
of the scattered wave; while the inverse scattering problem aims to recover the locally rough interface as well as the embedded
obstacle from the measured scattered wave in a certain domain. These problems was motivated by significant applications in 
diverse scientific areas such as medical imaging \cite{SA99} and exploration geophysics \cite{BA84}.

The main difficulty of the rough surface scattering problem is the unboundedness of the rough surface, which makes that the related 
integral operators are non-compact and so the classical Fredholm alternative is not available. Based on a generalized Fredholm theory
\cite{SZ97, SZ00}, the well-posedness of the direct scattering problem by rough surfaces has been established in \cite{SP05, SRZ99, SZ99,DTS03, ZS03} by the integral equation method. In addition, we refer to \cite{SE10, LWZ19, QZZ19, WW87, MT06, ZZ13} for the well-posedness of the direct rough surface scattering problem. Different from all previous works, if the rough surface is a local perturbation of a planar 
surface, by introducing a special locally rough surface, the scattering problem can be transformed into an equivalent integral equation 
defined in a bounded domain, for which the well-posedness follows from the classical Fredholm alternative, for details we refer to \cite{DLLY17, LYZ21} for the scattering by the locally rough surface and \cite{LYZ22, YLZ22} for the scattering by the locally rough interface with an embedded obstacle. For inverse problems, the reference \cite{YLZ22} has established a global uniqueness which shows that the locally rough interface, the embedded obstacle and the wave number in the lower half-space can be uniquely determined by means of near-field measurements above the interface. Based on this uniqueness, a modified linear sampling method has been developed in \cite{LYZ22} to solve the inverse problem of simultaneously reconstructing the locally rough interface and the embedded obstacle. However, it is worth pointing out that the linear sampling method in \cite{LYZ22} is sensitive to the noise. For the inverse problem, there exists a variety of numerical algorithms. 
For the inverse scattering by a planar surface with buried objects,
we refer to the MUSIC-type scheme \cite{AIL05}, the asymptotic factorization method \cite{RG08}, the sampling method \cite{GHKMS05}, and the direct imaging algorithm \cite{LLLL15, LYZZ21}. For the inverse scattering by rough surfaces, we refer to iterative algorithms \cite{GJ11,GJ13,CR10,CG11}, the algorithm based on the transformed field expansion \cite{GL13, GL14}, the factorization method \cite{AL08}, the singular source method \cite{C03}, the direct imaging method \cite{LZZ18,LZZ19}, and linear sampling methods \cite{DLLY17,LYZ21,ZY22}.

The reverse time migration (RTM) method is a popular sample-type method which has been extensively applied in seismic imaging \cite{BCS01} and exploration geophysics \cite{BA84,BCS01,JFC85}. This method first back-propagates the complex conjugated 
data into the background medium and then computes the cross-correlation between the back-propagated field and the incident field
to output the imaging indicator, which has different behaviors when the sampling point is near the scatterer and far away from the scatterer. Based on this property, it is efficient, stable, and robust to noise. The key point of the RTM method is to establish the related 
Helmholtz-Kirchhoff identity, which plays a crucial role in the analysis of the indicator. For the inverse obstacle scattering problem, the 
mathematical justification of the RTM method has been proved rigorously in \cite{CCH131}, which is based on an usual Helmholtz-Kirchhoff identity associated with the fundamental solution of the Helmholtz operator in the free space. The results in \cite{CCH131} has
been extended to \cite{CCH132,CH151,CH152,CH153,CH16,CH17, L21} to solve some other inverse obstacle scattering problems.
However, there are few results for the RTM method to recover unbounded rough surfaces since the usual Helmholtz-Kirchhoff identity
in \cite{CCH131} is not valid. To overcome this difficulty, we has established a modified Helmholtz-Kirchhoff identity associated with a background Green function in a two-layered medium separated by a special
locally rough surface and then extended the RTM method to reconstruct the locally rough surface without embedded obstacles, 
see \cite{LY22} for details. 

Unfortunately, if there is a obstacle embedded in the lower half-space separated by a locally rough interface, both the usual Helmholtz-Kirchhoff identity in \cite{CCH131} and the modified Helmholtz-Kirchhoff identity in \cite{LY22} are not valid to simultaneously reconstruct the locally rough interface and the embedded obstacle. Notice that the rough interface is locally perturbed, we are able to introduce a planar interface and then transform the scattering problem by the locally rough interface and the embedded obstacle into the scattering problem by several inhomogeneous mediums and the embedded obstacle. Based on this observation, we can establish the corresponding Helmholtz-Kirchhoff identity associated with the background Green function in a two-layered medium separated by the planar interface. Thus, the mathematical justification of the RTM method has been proved rigorously, where we illustrate that the corresponding imaging indicator always peaks on the local perturbation of the interface and on the embedded obstacle, which is confirmed in the numerical examples.

The outline of this paper is as follows. In section 2, we introduce the mathematical model for the forward scattering problem, and present a novel version of the RTM method based on a modified Helmholtz-Kirchhoff identity. In section 3, numerical examples are reported to illustrate the effectiveness of the proposed method. The paper is concluded with some general remarks and discussions on the future research in section 4.

\section{The RTM method}
In this section, we shall introduce the mathematical model of the scattering problem by an unbounded, locally rough interface and an 
obstacle in the lower half-space, and investigate the RTM method for this model.

Let the scattering interface be described by a curve 
\begin{eqnarray}\label{a1}
\Gamma:=\{(x_1,x_2)\in\mathbb R^2: x_2=f(x_1)\}
\end{eqnarray}
where $f$ is assumed to be a Lipschitz continuous function with compact support. It means that the interface $\Gamma$ is a local 
perturbation of the planar interface 
\begin{eqnarray*}\label{a2}
\Gamma_0:=\{(x_1,x_2)\in\mathbb R^2: x_2=0\}.
\end{eqnarray*}
The whole space is separated by $\Gamma$ into two half-spaces denoted by
\begin{eqnarray*}\label{a3}
\Omega_1:=\{(x_1,x_2)\in\mathbb R^2: x_2>f(x_1)\}\quad{\rm and}\quad \Omega_2:=\{(x_1,x_2)\in\mathbb R^2: x_2<f(x_1)\}.
\end{eqnarray*}
In the lower half-space $\Omega_2$, we assume that there is an bounded obstacle $D$ with a smooth boundary $\partial D\in C^{2,\alpha}$ for some H\"{o}lder exponent $0<\alpha\leq 1$. In this paper, for simplicity, $D$ is assumed to be sound-soft which 
means that a Dirichlet boundary condition is imposed on $\partial D$. 

Consider the incident field $u^i(x,x_s)$ to be generated by a point source 
\begin{eqnarray}\label{a4}
u^i(x,x_s)=\Phi_{\kappa}(x,x_s):=\frac{\rm i}{4}H_0^{(1)}(\kappa|x-x_s|)\quad {\rm for}\;\; x_s\in\R^2\setminus(\overline{D}\cup\Gamma).
\end{eqnarray}
Here, $H_0^{(1)}$ is the Hankel function of the first kind of order zero, $\kappa$ is the wave number which is a piecewise constant 
given by $\kappa=\kappa_1$ for $x_s\in\Omega_1$ and $\kappa=\kappa_2$ for $x_s\in\Omega_2\setminus\overline{D}$, and for some 
constant $\lambda$, $\Phi_{\lambda}$ is the fundamental solution of the Helmholtz equation satisfying $\Delta\Phi_{\lambda}(x, x_s)+\lambda^2\Phi_{\lambda}(x, x_s)=-\delta_{x_s}(x)$ in $\mathbb R^2$ in the distributional sense, where $\delta_{x_s}(x):=\delta(x-x_s)$ is 
the Kronecker delta distribution. Then the scattering of $u^i$ by the scatterer $(\Gamma, D)$ can be reduced to the problem of seeking 
$u=u(x,x_s)$ satisfying that 
\begin{equation}\label{a5}
\left\{\begin{array}{lll}
         \Delta u+\kappa^2u=-\delta_{x_s} \qquad\qquad&\textrm{in}\;\; \R^2\setminus\overline{D}, \\[2mm]
         u=0\qquad\qquad&\textrm{on}\;\; \partial D,\\[2mm]
          \lim\limits_{|x|\rightarrow \infty}|x|^{\frac{1}{2}}\left(\partial_{|x|} u^s-{\rm i}\kappa u^s\right)=0,
       \end{array}
\right.
\end{equation}
in the distributional sense. Here, $u(x,x_s)$ and $u^s(x,x_s)$ are the total field and the scattered field, respectively, which are related by
\begin{equation*}\label{a6}
u(x,x_s)=\left\{\begin{array}{ll}
\Phi_{\kappa_1}(x,x_s)+u^s(x,x_s)\qquad&\textrm{in}\;\;\Omega_1,\\[2mm]
u^s(x,x_s)\qquad\qquad&\textrm{in}\;\;\Omega_2\setminus\overline{D},
\end{array}
\right.
\end{equation*}
for $x_s\in\Omega_1$, and 
\begin{equation*}\label{a7}
u(x,x_s)=\left\{\begin{array}{ll}
u^s(x,x_s)\qquad&\textrm{in}\;\;\Omega_1,\\[2mm]
\Phi_{\kappa_2}(x,x_s)+u^s(x,x_s)\qquad&\textrm{in}\;\;\Omega_2\setminus\overline{D},
\end{array}
\right.
\end{equation*}
for $x_s\in\Omega_2\setminus\overline{D}$, the wave number $\kappa=\kappa(x)$ is a piecewise constant defined by
\begin{equation*}\label{a8}
\kappa(x)=\left\{\begin{array}{ll}
\kappa_1\qquad&\textrm{for}\;\;x\in\Omega_1,\\[2mm]
\kappa_2\qquad\qquad&\textrm{for}\;\;x\in\Omega_2\setminus\overline{D},
\end{array}
\right.
\end{equation*}
and the last condition in Problem (\ref{a5}) is the well-known Sommerfeld radiation condition which holds uniformly for all 
directions $\hat{x}:=x/|x|\in {\mathbb S}:=\{x\in\R^2: |x|=1\}$. Moreover, the Sommerfeld radiation condition allows that the 
scattered field $u^s$ has the following asymptotic behavior 
\begin{eqnarray*}\label{a9}
u^s(x,x_s)=\frac{e^{{\rm i}\kappa |x|}}{|x|^{\frac{1}{2}}}\left\{u^{\infty}(\hat{x},x_s)+O\left(\frac{1}{|x|}\right)\right\}\quad {\rm for}\; |x|\to\infty
\end{eqnarray*}
uniformly in all direction $\hat{x}\in{\mathbb S}$, where $u^{\infty}(\hat{x}, x_s)$ is known as the far-field pattern of $u^s(x,x_s)$. 
It is shown by Theorem 3.4 in \cite{YLZ22} that Problem (\ref{a5}) is well-posed.

\begin{figure}[htbp]
\centering
\includegraphics[width=5in, height=3in]{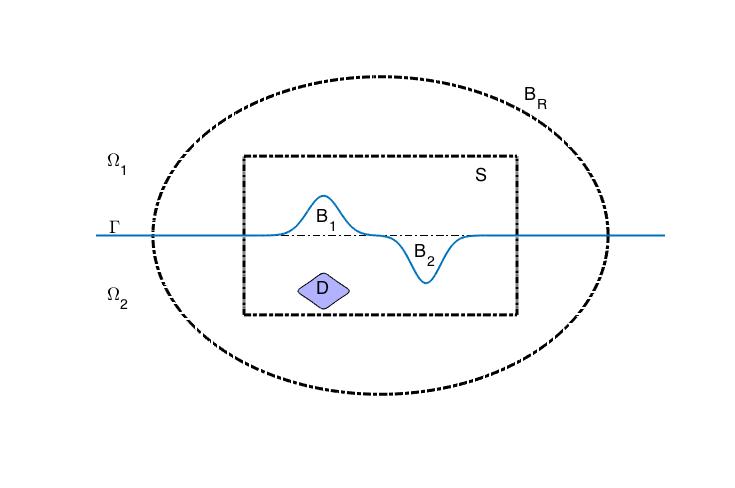}
\caption{The setting of RTM method for reconstruction of $(\Gamma, D)$.}
\label{f1} 
\end{figure}

As shown in Figure \ref{f1}, we first introduce some notations used in the RTM approach. Let $\Gamma$ be the locally rough interface 
given by (\ref{a1}), whose local perturbations are denoted by $B_1:=\R^2_+\cap\Omega_2$ and $B_2:=\R^2_-\cap\Omega_1$ with 
$\R^2_{\pm}:=\{x\in\R^2: x_2\gtrless 0\}$. For simplicity, in this paper we consider a simple case that $\Gamma$ has only two local perturbations as seen in Figure \ref{f1}. The results obtained in this paper is also valid for the case of multiple local perturbations. For 
$j=1,2$, let $\chi_j(x)$ denote the characterization function of the domain $B_j$, defined by $\chi_j(x)=1$ in $B_j$ and $\chi_j(x)=0$ outside of $B_j$. Throughout, denoted by $B:=B_1\cup B_2$ and $\chi(x):=\chi_1(x)-\chi_2(x)$. Let $D$ be the obstacle embedded in $\Omega_2$ and we assume a priori that $D\cap B=\emptyset$. Since the RTM method is a sample-type method, we choose a rectangle 
sampling domain $S$ such that $B\subset S$ and $D\subset S$. And we choose a sufficiently large $R$ such that $S\subset B_R$, where $B_R$ stands for the disc with the origin as the center and $R$ as the radius. Assume that there are $N_s$ point sources $x_s$ uniformly distributed on $\Gamma_s:=\partial B_R$ and $N_r$ receivers $x_r$ uniformly distributed on $\Gamma_r:=\partial B_R$.

To introduce the RTM method, we first consider the scattering of the incident point source $u^i(x,x_s)$ given by (\ref{a4}) by the planar 
interface $\Gamma_0$, which reads 
\begin{equation}\label{a10}
\left\{\begin{array}{ll}
         \Delta G(x,x_s)+\kappa_0^2(x)G(x,x_s)=-\delta_{x_s}(x)  \qquad\qquad&\textrm{in}\;\; \R^2, \\[2mm]
          \lim\limits_{|x|\rightarrow \infty}|x|^{\frac{1}{2}}\left(\partial_{|x|} G(x,x_s)-{\rm i}\kappa_0(x) G(x,x_s)\right)=0,
       \end{array}
\right.
\end{equation}
in the distributional sense with the Sommerfeld radiation condition uniformly for all $\hat{x}\in {\mathbb S}$. Here, the wave number
$\kappa_0(x)$ is defined by $\kappa_0(x)=\kappa_1$ in $\R^2_+$ and $\kappa_0(x)=\kappa_2$ in $\R^2_-$. We refer to \cite{L10, YLZ22} for the explicit expression of $G(x,x_s)$.

For $x_s\in\Gamma_s$, we define 
\begin{equation}\label{a11}
V(x,x_s):=u(x,x_s)-G(x,x_s),
\end{equation}
then it follows from (\ref{a5}) and (\ref{a10}) that $V(x,x_s)$ solves 
\begin{equation}\label{a12}
\left\{\begin{array}{lll}
         \Delta V(x,x_s)+\kappa^2(x)V(x,x_s)=\beta\chi(x)G(x,x_s)\qquad\qquad&\textrm{in}\;\; \R^2\setminus\overline{D}, \\[2mm]
         V(x,x_s)=-G(x,x_s)\qquad\qquad&\textrm{on}\;\; \partial D,\\[2mm]
          \lim\limits_{|x|\rightarrow \infty}|x|^{\frac{1}{2}}\left(\partial_{|x|} V(x,x_s)-{\rm i}\kappa(x) V(x,x_s)\right)=0,
       \end{array}
\right.
\end{equation}
where $\beta:=\kappa_1^2-\kappa_2^2$. We refer to Theorem 3.1 in \cite{LYZ22} for the unique solvability of Problem (\ref{a12}).

Since we can obtain $G(x_r,x_s)$ by solving Problem (\ref{a10}) through the Nystr\"{o}m method or the finite element method,
we can obtain the data $V(x_r,x_s)$ from the measurement $u(x_r,x_s)$ and (\ref{a11}). As mentioned in the introduction, we first
back-propagate the complex conjugated data $\overline{V(x_r,x_s)}$ into the domain $\R^2$, and then define the indicator as the 
imaginary part of the cross-correlation of $G(\cdot,x_s)$ and the back-propagation field. More precisely, we summarize it in the following 
algorithm.

{\bf Algorithm 1 (RTM for locally rough interfaces and embedded obstacles)}: Given the data $V(x_r,x_s)$ for $r=1,2,...,N_r$ and $s=1,2,...,N_s$.
\begin{itemize}
\item Back-propagation: for $s=1,2,...,N_s$, solve the problem 
\begin{eqnarray*}\label{a13}
\left\{\begin{aligned}
         &\Delta W(x,x_s)+\kappa_0^2(x)W(x,x_s)=\frac{|\Gamma_r|}{N_r}\sum_{r=1}^{N_r}\overline{V(x_r,x_s)}\delta_{x_r}(x) \qquad\textrm{in}\;\; \R^2, \\
        &\; \ds\lim_{|x|\rightarrow \infty}|x|^{\frac{1}{2}}\left(\partial_{|x|} W(x,x_s)-{\rm i}\kappa_0(x)  W(x,x_s)\right)=0,
       \end{aligned}
\right.
\end{eqnarray*}
to obtain the solution $W(x,x_s)$.
\item Cross-correlation: for each sampling point $z\in S$, compute the indicator function 
\begin{eqnarray*}\label{a14}
{\rm Ind}(z)=\kappa(x_r){\rm Im}\left\{\frac{|\Gamma_s|}{N_s}\sum_{s=1}^{N_s}\kappa(x_s)G(z,x_s)W(z,x_s)\right\}
\end{eqnarray*}
and then plot the mapping ${\rm Ind}(z)$ against $z$.
\end{itemize}

From Problem (\ref{a10}) and the linearity, we immediately see that 
\begin{eqnarray*}\label{a15}
W(x,x_s)=-\frac{|\Gamma_r|}{N_r}\sum_{r=1}^{N_r}G(x,x_r)\overline{V(x_r,x_s)},
\end{eqnarray*}
which implies that 
\begin{eqnarray}\label{a16}
{\rm Ind}(z)=-{\rm Im}\left\{\frac{|\Gamma_s|}{N_s}\frac{|\Gamma_r|}{N_r}\sum_{s=1}^{N_s}\sum_{r=1}^{N_r}\kappa(x_r)\kappa(x_s)G(z,x_s)G(z,x_r)\overline{V(x_r,x_s)}\right\}\quad z\in S.
\end{eqnarray}
Observing that $G(z,x_s)$, $G(z,x_r)$ and $V(x_r, x_s)$ are continuous for $z\in S$, $x_r\in\Gamma_r$, $x_s\in\Gamma_s$, it follows
from the trapezoid quadrature formula that ${\rm Ind}(z)$ given by (\ref{a16}) is a discrete formula of the following continuous function
\begin{eqnarray}\label{a17}
\widetilde{{\rm Ind}}(z)=-{\rm Im}\int_{\Gamma_r}\int_{\Gamma_s}\kappa(x_r)\kappa(x_s)G(z,x_s)G(z,x_r)\overline{V(x_r,x_s)}{\rm d}s(x_s){\rm d}s(x_r),\quad z\in S.
\end{eqnarray}

In the remaining part of this section, we restrict ourselves to demonstrate that the indicator $\widetilde{{\rm Ind}}(z)$ enjoys the nice 
feature that it always peaks on the local perturbation $B$ and on the embedded obstacle $D$. To this end, we first introduce the following modified Helmholtz-Kirchhoff identity associated with the Green function $G$.
\begin{lemma}\label{lem1}
Let $G$ be the background Green's function defined by (\ref{a10}). Then for any $x,z\in B_R\setminus\Gamma_0$, we have 
\begin{eqnarray}\label{a18}
\int_{\partial B_R}\left(\overline{G(\xi,x)}\frac{\partial G(\xi,z)}{\partial \nu(\xi)}-\frac{\partial \overline{G(\xi,x)}}{\partial\nu(\xi)}G(\xi,z)\right){\rm d}s(\xi)=2{\rm i}{\rm Im}G(x,z).
\end{eqnarray}
\end{lemma}
\begin{proof}
For any $x,z\in B_R\setminus\Gamma_0$, we choose a sufficiently small $\varepsilon>0$ such that the circles $B_{\varepsilon}(x)$, $B_{\varepsilon}(z)$ with $x,z$ as the center and $\varepsilon$ as the radius contain in the domain $B_R$ and $B_{\varepsilon}(x)\cap\Gamma_0=\emptyset$, $B_{\varepsilon}(z)\cap\Gamma_0=\emptyset$. Note that $G$ and its normal derivative are continuous across 
the interface $\Gamma_0$, thus a direct application of the Green theorem to $\overline{G(\cdot,x)}$ and $G(\cdot,z)$ in the domain $(B_R^+\cup B_R^-)\setminus(B_{\varepsilon}(x)\cup B_{\varepsilon}(z))$ yields 
\begin{eqnarray*}\nonumber
0&=&\int_{(B_R^+\cup B_R^-)\setminus(B_{\varepsilon}(x)\cup B_{\varepsilon}(z))}\left(\overline{G(\xi,x)}\Delta G(\xi,z)-\Delta\overline{G(\xi,x)}G(\xi,z)\right){\rm d}\xi\\\nonumber
&=&\int_{\partial B_R\cup\partial B_{\varepsilon}(x)\cup\partial B_{\varepsilon}(z)}\left(\overline{G(\xi,x)}\frac{\partial G(\xi,z)}{\partial \nu(\xi)}-\frac{\partial\overline{G(\xi,x)}}{\partial \nu(\xi)}G(\xi,z)\right){\rm d}s(\xi)\\\label{a19}
&:=&I_1+I_2+I_3
\end{eqnarray*}
where $B_R^+:=B_R\cap\R^2_+$, $B_R^-:=B_R\cap\R^2_-$, $\nu(\xi)$ denotes the unit outward normal to $\partial B_R$ when $\xi\in\partial B_R$, $\nu(\xi)$ denotes the unit normal to $\partial B_{\varepsilon}(y)$ into the interior of $B_{\varepsilon}(y)$ for $y\in\{x,z\}$. For the case $x\in\R^2_+$, notice that $G(\xi,x):=\Phi_{\kappa_1}(\xi,x)+G^s(\xi,x)$, we have 
\begin{eqnarray*}\nonumber
I_2 &=& \int_{\partial B_{\varepsilon}(x)}\left(\overline{\Phi_{\kappa_1}(\xi,x)}\frac{\partial G(\xi,z)}{\partial \nu(\xi)}-\frac{\partial\overline{\Phi_{\kappa_1}(\xi,x)}}{\partial \nu(\xi)}G(\xi,z)\right){\rm d}s(\xi)\\\nonumber
&&+\int_{\partial B_{\varepsilon}(x)}\left(\overline{G^s(\xi,x)}\frac{\partial G(\xi,z)}{\partial \nu(\xi)}-\frac{\partial\overline{G^s(\xi,x)}}{\partial \nu(\xi)}G(\xi,z)\right){\rm d}s(\xi)\\\label{a20}
&\to&-G(x,z),\quad{\rm as}\;\;\varepsilon\to 0.
\end{eqnarray*}
For the case $x\in\R^2_-$, by a similar argument, we obtain $I_2\to -G(x,z)$ as $\varepsilon\to 0$. Similarly, we have
\begin{eqnarray*}\label{a21}
\lim_{\varepsilon\to 0}I_3=\overline{G(z,x)}.
\end{eqnarray*}
Thus, we obtain
\begin{eqnarray*}
\int_{\partial B_R}\left(\overline{G(\xi,x)}\frac{\partial G(\xi,z)}{\partial \nu(\xi)}-\frac{\partial\overline{G(\xi,x)}}{\partial \nu(\xi)}G(\xi,z)\right){\rm d}s(\xi)=G(x,z)-\overline{G(z,x)}=2{\rm i}{\rm Im} G(x,z).
\end{eqnarray*}
Here, we use the reciprocity $G(x,z)=G(z,x)$ for $x,z\in B_R\setminus\Gamma_0$, which can be proven by a direct application of the 
Green theorem and the Sommerfeld radiation condition. The proof is completed.
\end{proof}
 
With the aid of the modified Helmholtz-Kirchhoff identity (\ref{a18}), we can obtain the following lemma which plays 
an important role in the analysis of $\widetilde{{\rm Ind}}(z)$.
\begin{lemma}\label{lem2}
For any $x,z\in S$, we have
\begin{eqnarray*}\label{a22}
\int_{\partial B_R}\kappa(\xi)\overline{G(x,\xi)}G(\xi,z){\rm d}s(\xi)={\rm Im}G(x,z)+\zeta(x,z)
\end{eqnarray*}
where $\zeta(x,z)$ satisfies
\begin{eqnarray*}\label{a23}
|\zeta(x,z)|+|\nabla_x\zeta(x,z)|\leq CR^{-1}
\end{eqnarray*}
uniformly for any $x,z\in S$.
\end{lemma}
\begin{proof}
For $x,z\in S$, using Lemma \ref{lem1}, we have
\begin{eqnarray*}
2{\rm i}{\rm Im}G(x,z)&=&\int_{\partial B_R}\left(\overline{G(\xi,x)}\frac{\partial G(\xi,z)}{\partial \nu(\xi)}-\frac{\partial \overline{G(\xi,x)}}{\partial\nu(\xi)}G(\xi,z)\right){\rm d}s(\xi)\\
&=&\int_{\partial B_R}\Bigg\{\overline{G(\xi,x)}\left[\frac{\partial G(\xi,z)}{\partial \nu(\xi)}-{\rm i}\kappa(\xi) G(\xi,z)\right]{\rm d}s(\xi)\\
&&-G(\xi,z)\left[\frac{\overline{\partial G(\xi,x)}}{\partial \nu(\xi)}+{\rm i}\kappa(\xi)\overline{G(\xi,x)}\right]\Bigg\}{\rm d}s(\xi)\\
&&+2{\rm i}\int_{\partial B_R}\kappa(\xi)\overline{G(\xi,x)}G(\xi,z){\rm d}s(\xi),
\end{eqnarray*}
which implies
\begin{eqnarray*}
\int_{\partial B_R}\kappa(\xi)\overline{G(x,\xi)}G(\xi,z){\rm d}s(\xi)={\rm Im}G(x,z)+\zeta(x,z)\quad{\rm for}\;\;{\forall x,z\in S},
\end{eqnarray*}
with 
\begin{eqnarray*}
\zeta(x,z)&=&\frac{\rm i}{2}\int_{\partial B_R}\Bigg\{\overline{G(\xi,x)}\left[\frac{\partial G(\xi,z)}{\partial \nu(\xi)}-{\rm i}\kappa(\xi) G(\xi,z)\right]{\rm d}s(\xi)\\
&&\qquad\;\;\;\;\;-G(\xi,z)\left[\frac{\overline{\partial G(\xi,x)}}{\partial \nu(\xi)}+{\rm i}\kappa(\xi)\overline{G(\xi,x)}\right]\Bigg\}{\rm d}s(\xi).
\end{eqnarray*}
For $y\in\{x,z\}$, since
\begin{eqnarray*}
G(\xi,y)=O(|\xi|^{-\frac{1}{2}}),\qquad\frac{\partial G(\xi,y)}{\partial \nu(\xi)}-{\rm i}\kappa(\xi) G(\xi,y)=O(|\xi|^{-\frac{3}{2}})
\end{eqnarray*}
we conclude that 
\begin{eqnarray*}
|\zeta(x,z)|\leq CR^{-1}
\end{eqnarray*}
uniformly for $x,z\in S$. Moreover, due to 
\begin{eqnarray*}
\frac{\partial G(\xi,x)}{\partial x_j}=O(|\xi|^{-\frac{1}{2}}),\qquad\frac{\partial}{\partial x_j}\left[\frac{\partial G(\xi,x)}{\partial \nu(\xi)}-{\rm i}\kappa G(\xi,x)\right]=O(|\xi|^{-\frac{3}{2}})
\end{eqnarray*}
for $j=1,2$, we obtain 
\begin{eqnarray*}
|\nabla_x\zeta(x,z)|\leq CR^{-1}
\end{eqnarray*}
uniformly for $x,z\in S$. The proof is finished.
\end{proof}

\begin{lemma}\label{lem3}
Let $V(x,x_s)$ be the solution of Problem (\ref{a12}). Then we have
\begin{eqnarray*}\nonumber
V(x,x_s)=\int_{\partial D}\left[\frac{\partial V(\xi,x_s)}{\partial\nu(\xi)}G(\xi,x)-\frac{\partial G(\xi,x)}{\partial \nu(\xi)}V(\xi,x_s)\right]{\rm d}s(\xi)-\beta\int_B\chi(\xi)G(\xi,x)u(\xi,x_s){\rm d}\xi
\end{eqnarray*}
for $x\in\R^2\setminus(\Gamma\cup\overline{D})$, where $\nu(\xi)$ denotes the unit normal at $\xi\in\partial D$ which directs into the interior of $D$.
\end{lemma}
\begin{proof}
For any $x\in\R^2\setminus(\Gamma\cup\overline{D})$, we choose a sufficiently small $\varepsilon>0$ and a sufficiently large $\rho>0$ such that the disc $B_{\varepsilon}(x)$ with $x$ as the center and $\varepsilon$ as the radius contains in the domain $\Omega_{\rho}:=B_{\rho}\setminus(\Gamma\cup\overline{D})$, where $B_{\rho}$ stands for the disc with the origin as the center and $\rho$ as the radius. Using the Green's theorem to the functions $V(\xi,x_s)$ and $G(\xi,x)$ in the domain $\Omega_{\rho}\setminus B_{\varepsilon}(x)$, we have
\begin{eqnarray}\nonumber
&&\int_{\Omega_{\rho}\setminus B_{\varepsilon}(x)}\left[\Delta V(\xi,x_s)G(\xi,x)-\Delta G(\xi,x)V(\xi,x_s)\right]{\rm d}\xi\\\nonumber
&&=\int_{\partial B_{\rho}\cup\partial D\cup\partial B_{\varepsilon}(x)}\left[\frac{\partial V(\xi,x_s)}{\partial \nu(\xi)}G(\xi,x)-\frac{\partial G(\xi,x)}{\partial\nu(\xi)}V(\xi,x_s)\right]{\rm d}s(\xi)\\\label{a24}
&&:=J_1+J_2+J_3,
\end{eqnarray}
where $\nu(\xi)$ denotes the unit normal which directs into the interior of $B_{\varepsilon}(x)$ when $\xi\in\partial B_{\varepsilon}(x)$, and directs into the exterior of $B_{\rho}$ when $\xi\in\partial B_{\rho}$. By (\ref{a10}) and (\ref{a12}), it is easily seen that 
\begin{eqnarray}\label{a25}
\lim_{\varepsilon\to 0}\int_{\Omega_{\rho}\setminus B_{\varepsilon}(x)}\left[\Delta V(\xi,x_s)G(\xi,x)-\Delta G(\xi,x)V(\xi,x_s)\right]{\rm d}\xi=\beta\int_B\chi(\xi)G(\xi,x)u(\xi,x_s){\rm d}\xi.
\end{eqnarray}
For the item $J_1$, we have
\begin{eqnarray*}\nonumber
J_1&=&\int_{\partial B_{\rho}}\left[\frac{\partial V(\xi,x_s)}{\partial\nu(\xi)}-{\rm i}\kappa(\xi) V(\xi,x_s)\right]G(\xi,x){\rm d}s(\xi)\\\nonumber
&&-\int_{\partial B_{\rho}}\left[\frac{\partial G(\xi,x)}{\partial\nu(\xi)}-{\rm i}\kappa(\xi) G(\xi,x)\right]V(\xi,x_s){\rm d}s(\xi)\\\label{a26}
&:=&J_{11}-J_{12}.
\end{eqnarray*}
Since 
\begin{eqnarray*}\label{a27}
G(\xi,x)=O(|\xi|^{-\frac{1}{2}}),\qquad\frac{\partial V(\xi,x_s)}{\partial \nu(\xi)}-{\rm i}\kappa(\xi) V(\xi,x_s)=O(|\xi|^{-\frac{3}{2}})
\end{eqnarray*}
it follows from the Cauchy-Schwarz inequality that 
\begin{eqnarray*}\label{a28}
\lim_{\rho\to\infty}J_{11}=0.
\end{eqnarray*}
For the item $J_{12}$, we claim that 
\begin{eqnarray}\label{a29}
\int_{\partial B_{\rho}}|V(\xi,x_s)|^2{\rm d}s(\xi)=O(1)\quad{\rm as}\;\;\rho\to\infty.
\end{eqnarray}
To show this, it is found by the Sommerfeld radiation condition that 
\begin{eqnarray}\label{a30}
\int_{\partial B_{\rho}}\left[\left|\frac{\partial V(\xi,x_s)}{\partial\nu(\xi)}\right|^2+\kappa^2(\xi)|V(\xi,x_s)|^2+2\kappa(\xi){\rm Im}\left(V(\xi,x_s)\frac{\partial\overline{V(\xi,x_s)}}{\partial\nu(\xi)}\right)\right]{\rm d}s(\xi)\to 0
\end{eqnarray}
as $\rho\to\infty$. Applying the Green theorem for $V(\xi,x_s)$ and $\overline{V(\xi,x_s)}$ in the domain $B_{\rho}\setminus\overline{D}$ 
implies that 
\begin{eqnarray}\nonumber
&&\int_{\partial B_{\rho}}V(\xi,x_s)\frac{\partial \overline{V(\xi,x_s)}}{\partial\nu(\xi)}{\rm d}s(\xi)\\\nonumber
&=&\int_{B_{\rho}\setminus\overline{D}}\left(|\nabla V(\xi,x_s)|^2-\kappa^2(\xi)|V(\xi,x_s)|^2\right){\rm d}\xi\\\label{a31}
&-&\int_{\partial D}V(\xi,x_s)\frac{\partial \overline{V(\xi,x_s)}}{\partial\nu(\xi)}{\rm d}s(\xi)+\beta\int_{B}\chi(\xi)V(\xi,x_s)\overline{G(\xi,x_s)}{\rm d}\xi.
\end{eqnarray}
Taking the imaginary part of (\ref{a31}) and substituting it to (\ref{a30}) gives that 
\begin{eqnarray*}
&&\lim_{\rho\to\infty}\frac{1}{2}\int_{\partial B_{\rho}}\left(\frac{1}{\kappa(\xi)}\left|\frac{\partial V(\xi,x_s)}{\partial\nu(\xi)}\right|^2+\kappa(\xi)|V(\xi,x_s)|^2\right){\rm d}s(\xi)\\
&&={\rm Im}\left\{\int_{\partial D}V(\xi,x_s)\frac{\partial \overline{V(\xi,x_s)}}{\partial\nu(\xi)}{\rm d}s(\xi)-\beta\int_{B}\chi(\xi)V(\xi,x_s)\overline{G(\xi,x_s)}{\rm d}\xi\right\}
\end{eqnarray*}
which implies that the claim (\ref{a29}) holds true. Since
\begin{eqnarray}\label{a32}
\frac{\partial G(\xi,x)}{\partial \nu(\xi)}-{\rm i}\kappa(\xi) G(\xi,x)=O(|\xi|^{-\frac{3}{2}}),
\end{eqnarray}
it follows from (\ref{a29}), (\ref{a32}), and the Cauchy-Schwarz inequality that $\lim\limits_{\rho\to\infty}J_{12}=0$. Hence, we obtain
\begin{eqnarray}\label{a33}
\lim_{\rho\to\infty}J_{1}=0.
\end{eqnarray}

For the item $J_3$, if $x\in\R^2_+$, using $G(\xi,x)=\Phi_{\kappa_1}(\xi,x)+G^s(\xi,x)$, we have 
\begin{eqnarray*}\nonumber
J_3 &=& \int_{\partial B_{\varepsilon}(x)}\left(\Phi_{\kappa_1}(\xi,x)\frac{\partial V(\xi,x_s)}{\partial \nu(\xi)}-\frac{\partial\Phi_{\kappa_1}(\xi,x)}{\partial \nu(\xi)}V(\xi,x_s)\right){\rm d}s(\xi)\\\nonumber
&&+\int_{\partial B_{\varepsilon}(x)}\left(G^s(\xi,x)\frac{\partial V(\xi,x_s)}{\partial \nu(\xi)}-\frac{\partial G^s(\xi,x)}{\partial \nu(\xi)}V(\xi,x_s)\right){\rm d}s(\xi)\\\label{a34}
&:=&J_{31}+J_{32}.
\end{eqnarray*}
A direct calculation, using the mean value theorem, shows that
\begin{eqnarray*}\label{a35}
\lim_{\varepsilon\to 0}J_{31}=-V(x,x_s).
\end{eqnarray*}
For the item $J_{32}$, if $x\in\R^2_+\cap\Omega_1$, by the Green theorem, (\ref{a10}) and (\ref{a12}), we have
\begin{eqnarray*}
J_{32}=\int_{B_{\varepsilon}(x)}\left(\Delta V(\xi, x_s)G^s(\xi,x)-\Delta G^s(\xi,x)V(\xi,x)\right){\rm d}\xi=0.
\end{eqnarray*}
If $x\in B_1$, combining the Green theorem, (\ref{a10}) and (\ref{a12}) shows that
\begin{eqnarray*}
J_{32}&=&\int_{B_{\varepsilon}(x)}\left(\Delta V(\xi, x_s)G^s(\xi,x)-\Delta G^s(\xi,x)V(\xi,x)\right){\rm d}\xi\\
&=&\beta\int_{B_{\varepsilon}(x)}G^s(\xi,x)u(\xi,x_s){\rm d}\xi\to 0 \quad{\rm as}\;\;\varepsilon\to 0.
\end{eqnarray*}
If $x\in\R^2_-$, by a similar argument, we can obtain the same result. Thus, we conclude that 
\begin{eqnarray}\label{a36}
\lim_{\varepsilon\to 0}J_3=-V(x,x_s).
\end{eqnarray}
It is found by (\ref{a24}), (\ref{a25}), (\ref{a33}), and (\ref{a36}) that the Lemma holds true. The proof is completed.
\end{proof}

With the help of Lemma \ref{lem1}, Lemma \ref{lem2}, and Lemma \ref{lem3}, we are in position to present the main 
result of this paper, which focuses on the resolution analysis of the RTM approach for simultaneously reconstruction 
of the locally rough interface and the embedded obstacle.

\begin{theorem}\label{thm}
For any $z\in S$, let $\psi(\xi,z)$ be the solution of 
\begin{equation}\label{a37}
\left\{\begin{array}{lll}
         \Delta \psi(\xi,z)+\kappa^2(\xi)\psi(\xi,z)=\beta\chi(\xi){\rm Im}G(\xi,z)\qquad\qquad&{\rm in}\;\; \R^2\setminus\overline{D}, \\[2mm]
         \psi(\xi,z)=-{\rm Im}G(\xi,z)\qquad\qquad&{\rm on}\;\; \partial D,\\[2mm]
          \lim\limits_{|\xi|\rightarrow \infty}|\xi|^{\frac{1}{2}}\left(\partial_{|\xi|} \psi(\xi,z)-{\rm i}\kappa(\xi) \psi(\xi,z)\right)=0,
       \end{array}
\right.
\end{equation}
and $\psi^{\infty}(\hat{\xi},z)$ be the corresponding far-field pattern. Then we have
\begin{eqnarray*}\label{a38}
\widetilde{{\rm Ind}}(z)=\int_{{\mathbb S}}\kappa(\hat{\xi})|\psi^{\infty}(\hat{\xi},z)|^2{\rm d}s(\hat{\xi})+\eta(z),\qquad \forall z\in S,
\end{eqnarray*}
where $\|\eta(z)\|_{L^{\infty}(S)}\leq CR^{-1}$ with some constant $C$ depending on $B$ and $D$.
\end{theorem}
\begin{proof}
Define 
\begin{eqnarray}\label{a39}
\widetilde{W}(z,x_s)=\int_{\Gamma_r}\kappa(x_r)G(z,x_r)\overline{V(x_r,x_s)}{\rm d}s(x_r),
\end{eqnarray}
substituting the Green formula presented in Lemma \ref{lem3} into $\widetilde{W}(z,x_s)$ gives that 
\begin{eqnarray*}\label{a40}
\widetilde{W}(z,x_s)=\int_{\partial D}\left[\frac{\partial \overline{V(\xi,x_s)}}{\partial\nu(\xi)}\gamma(\xi,z)-\frac{\partial \gamma(\xi,z)}{\partial \nu(\xi)}\overline{V(\xi,x_s)}\right]{\rm d}s(\xi)-\beta\int_B\chi(\xi)\overline{u(\xi,x_s)}\gamma(\xi,z){\rm d}\xi
\end{eqnarray*}
with 
\begin{eqnarray}\label{a41}
\gamma(\xi,z):=\int_{\Gamma_r}\kappa(x_r)G(z,x_r)\overline{G(\xi,x_r)}{\rm d}s(x_r)={\rm Im}G(\xi,z)+\zeta(\xi,z)
\end{eqnarray}
where we use Lemma \ref{lem2}. According to (\ref{a17}) and (\ref{a39}), we obtain 
\begin{eqnarray}\nonumber
\widetilde{\rm Ind}(z)&=&-{\rm Im}\int_{\Gamma_s}\kappa(x_s)G(z,x_s)\widetilde{W}(z,x_s){\rm d}s(x_s)\\\nonumber
&=&-{\rm Im}\int_{\partial D}\left[\frac{\partial \phi_1(\xi,z)}{\partial\nu(\xi)}\gamma(\xi,z)-\frac{\partial \gamma(\xi,z)}{\partial \nu(\xi)}\phi_1(\xi,z)\right]{\rm d}s(\xi)\\\label{a42}
&&+{\rm Im}\beta\int_B\chi(\xi)\phi_2(\xi,z)\gamma(\xi,z){\rm d}\xi
\end{eqnarray}
where $\phi_1(\xi,z)$ and $\phi_2(\xi,z)$ are defined by
\begin{eqnarray*}\label{a43}
\phi_1(\xi,z):=\int_{\Gamma_s}\kappa(x_s)G(z,x_s)\overline{V(\xi,x_s)}{\rm d}s(x_s)\\\label{a44}
\phi_2(\xi,z):=\int_{\Gamma_s}\kappa(x_s)G(z,x_s)\overline{u(\xi,x_s)}{\rm d}s(x_s).
\end{eqnarray*}
By the identity $u(\xi,x_s)=V(\xi,x_s)+G(\xi,x_s)$ and Lemma \ref{lem2}, it is easily found that
\begin{eqnarray}\label{a45}
\phi_2(\xi,z)=\phi_1(\xi,z)+{\rm Im}G(\xi,z)+\overline{\zeta(\xi,z)}.
\end{eqnarray}
Since $V(\xi,x_s)$ satisfies Problem (\ref{a12}), it is easily checked that $\overline{\phi_1(\xi,z)}$ solves 
\begin{equation}\label{a46}
\left\{\begin{array}{lll}
         \Delta \overline{\phi_1(\xi,z)}+\kappa^2(\xi)\overline{\phi_1(\xi,z)}=\beta\chi(\xi)\left[{\rm Im}G(\xi,z)+\zeta(\xi,z)\right]\qquad\qquad&{\rm in}\;\; \R^2\setminus\overline{D}, \\[3mm]
         \overline{\phi_1(\xi,z)}=-\left[{\rm Im}G(\xi,z)+\zeta(\xi,z)\right]\qquad\qquad&{\rm on}\;\; \partial D,\\[3mm]
          \lim\limits_{|\xi|\rightarrow \infty}|\xi|^{\frac{1}{2}}\left(\partial_{|\xi|} \overline{\phi_1(\xi,z)}-{\rm i}\kappa(\xi) \overline{\phi_1(\xi,z)}\right)=0,
       \end{array}
\right.
\end{equation}
where we use Lemma \ref{lem2} to obtain the right hand term. Let $\psi(\xi,z)$ and $\varphi(\xi,z)$ solve Problem (\ref{a46}) expect the right hand term are replaced by $(\beta\chi(\xi){\rm Im}G(\xi,z), -{\rm Im}G(\xi,z))$ and $(\beta\chi(\xi)\zeta(\xi,z), -\zeta(\xi,z))$, respectively. Then by the linearity we have 
\begin{eqnarray}\label{a47}
\overline{\phi_1(\xi,z)}=\psi(\xi,z)+\varphi(\xi,z).
\end{eqnarray}
Substituting (\ref{a41}), (\ref{a45}) and (\ref{a47}) into (\ref{a42}) implies that 
\begin{align*}
\widetilde{\rm Ind}(z)&=-{\rm Im}\int_{\partial D}\bigg\{\left[{\rm Im}G(\xi,z)+\zeta(\xi,z)\right]\frac{\partial}{\partial\nu(\xi)}\left[\overline{\psi(\xi,z)}+\overline{\varphi(\xi,z)}\right]\\
&\quad -\left[\overline{\psi(\xi,z)}+\overline{\varphi(\xi,z)}\right]\frac{\partial}{\partial\nu(\xi)}\left[{\rm Im}G(\xi,z)+\zeta(\xi,z)\right]\bigg\}{\rm d}s(\xi)\\
&\quad+{\rm Im}\beta\int_B\chi(\xi)\left[\overline{\psi(\xi,z)}+\overline{\varphi(\xi,z)}+{\rm Im}G(\xi,z)+\overline{\zeta(\xi,z)}\right]\left[{\rm Im}G(\xi,z)+\zeta(\xi,z)\right]{\rm d}\xi\\
&=-{\rm Im}\int_{\partial D}\left[\frac{\partial \overline{\psi(\xi,z)}}{\partial\nu(\xi)}{\rm Im}G(\xi,z)-\frac{\partial{\rm Im}G(\xi,z)}{\partial\nu(\xi)}\overline{\psi(\xi,z)}\right]{\rm d}s(\xi)\\
&\quad+{\rm Im}\beta\int_B\chi(\xi)\overline{\psi(\xi,z)}{\rm Im}G(\xi,z){\rm d}\xi+\eta(z)\\
&={\rm Im}\int_{\partial D}\frac{\partial \overline{\psi(\xi,z)}}{\partial\nu(\xi)}\psi(\xi,z){\rm d}s(\xi)+{\rm Im}\beta\int_B\chi(\xi)\overline{\psi(\xi,z)}{\rm Im}G(\xi,z){\rm d}\xi+\eta(z),
\end{align*}
where $\eta(z)$ is defined by
\begin{eqnarray*}
\eta(z)&:=&-{\rm Im}\int_{\partial D}\left[\frac{\partial\overline{\psi(\xi,z)}}{\partial\nu(\xi)}\zeta(\xi,z)-\frac{\partial\zeta(\xi,z)}{\partial\nu(\xi)}\overline{\psi(\xi,z)}\right]{\rm d}s(\xi)\\
&&-{\rm Im}\int_{\partial D}\left[\frac{\partial\overline{\varphi(\xi,z)}}{\partial\nu(\xi)}{\rm Im}G(\xi,z)-\frac{\partial{\rm Im}G(\xi,z)}{\partial\nu(\xi)}\overline{\varphi(\xi,z)}\right]{\rm d}s(\xi)\\
&&-{\rm Im}\int_{\partial D}\left[\frac{\partial\overline{\varphi(\xi,z)}}{\partial\nu(\xi)}\zeta(\xi,z)-\frac{\partial\zeta(\xi,z)}{\partial\nu(\xi)}\overline{\varphi(\xi,z)}\right]{\rm d}s(\xi)\\
&&+{\rm Im}\beta\int_B\chi(\xi)\left[\overline{\psi(\xi,z)}\zeta(\xi,z)+\overline{\varphi(\xi,z)}{\rm Im}G(\xi,z)+\overline{\varphi(\xi,z)}\zeta(\xi,z)\right]{\rm d}s(\xi).
\end{eqnarray*}
We choose a sufficiently large $\rho>0$ such that $B\cup D\subset B_{\rho}$ and apply the Green theorem for $\psi(\xi,z)$ and $\overline{\psi(\xi,z)}$ in the domain $B_{\rho}\setminus\overline{D}$ to get 
\begin{eqnarray*}
&&\int_{B_{\rho}\setminus\overline{D}}\Delta\psi(\xi,z)\overline{\psi(\xi,z)}{\rm d}\xi\\
&&=\int_{B_{\rho}\setminus\overline{D}}\left[\beta\chi(\xi){\rm Im}G(\xi,z)-\kappa^2(\xi)\psi(\xi,z)\right]\overline{\psi(\xi,z)}{\rm d}\xi\\
&&=\int_{\partial B_{\rho}}\frac{\partial \psi(\xi,z)}{\partial\nu(\xi)}\overline{\psi(\xi,z)}{\rm d}s(\xi)+\int_{\partial D}\frac{\partial \psi(\xi,z)}{\partial\nu(\xi)}\overline{\psi(\xi,z)}{\rm d}s(\xi)-\int_{B_{\rho}\setminus\overline{D}}|\nabla\psi(\xi,z)|^2{\rm d}\xi.
\end{eqnarray*}
Inserting the imaginary part on both sides of the above equation gives that
\begin{eqnarray*}\label{a48}
{\rm Im}\left\{\int_{\partial D}\frac{\partial \overline{\psi(\xi,z)}}{\partial\nu(\xi)}\psi(\xi,z){\rm d}s(\xi)+\beta\int_B\chi(\xi)\overline{\psi(\xi,z)}{\rm Im}G(\xi,z){\rm d}\xi\right\}={\rm Im}\int_{\partial B_{\rho}}\frac{\partial \psi(\xi,z)}{\partial\nu(\xi)}\overline{\psi(\xi,z)}{\rm d}s(\xi).
\end{eqnarray*}
Thus, we obtain 
\begin{eqnarray*}
&&\widetilde{\rm Ind}(z)={\rm Im}\int_{\partial B_{\rho}}\frac{\partial \psi(\xi,z)}{\partial\nu(\xi)}\overline{\psi(\xi,z)}{\rm d}s(\xi)+\eta(z)\\
&&={\rm Im}\int_{\partial B_{\rho}}\left[\frac{\partial \psi(\xi,z)}{\partial\nu(\xi)}-{\rm i}\kappa(\xi)\psi(\xi,z)\right]\overline{\psi(\xi,z)}{\rm d}s(\xi)+\int_{\partial B_{\rho}}\kappa(\xi)|\psi(\xi,z)|^2{\rm d}s(\xi)+\eta(z).
\end{eqnarray*}
Since 
\begin{eqnarray*}
\psi(\xi,z)=O(|\xi|^{-\frac{1}{2}}),\qquad\frac{\partial \psi(\xi,z)}{\partial \nu(\xi)}-{\rm i}\kappa(\xi) \psi(\xi,z)=O(|\xi|^{-\frac{3}{2}})
\end{eqnarray*}
we have
\begin{eqnarray*}
\lim_{\rho\to\infty}\int_{\partial B_{\rho}}\left[\frac{\partial \psi(\xi,z)}{\partial\nu(\xi)}-{\rm i}\kappa(\xi)\psi(\xi,z)\right]\overline{\psi(\xi,z)}{\rm d}s(\xi)=0.
\end{eqnarray*}
It follows from the Sommerfeld radiation condition that 
\begin{eqnarray*}
\psi(\xi,z)=\frac{e^{{\rm i}\kappa |\xi|}}{|\xi|^{\frac{1}{2}}}\left\{\psi^{\infty}(\hat{\xi},z)+O\left(\frac{1}{|\xi|}\right)\right\}\quad {\rm for}\; |\xi|\to\infty
\end{eqnarray*}
which leads to 
\begin{eqnarray*}
\lim_{\rho\to\infty}\int_{\partial B_{\rho}}\kappa(\xi)|\psi(\xi,z)|^2{\rm d}s(\xi)=\int_{\mathbb S}\kappa(\hat{\xi})|\psi^{\infty}(\hat{\xi},z)|^2{\rm d}s(\hat{\xi}).
\end{eqnarray*}
Thus, we conclude that 
\begin{eqnarray*}
\widetilde{{\rm Ind}}(z)=\int_{{\mathbb S}}\kappa(\hat{\xi})|\psi^{\infty}(\hat{\xi},z)|^2{\rm d}s(\hat{\xi})+\eta(z),\qquad \forall z\in S.
\end{eqnarray*}

In the remaining part of the proof, we restrict ourselves to the estimate of $\eta(z)$. For $\xi\in S$, $z\in S$, due to Lemma \ref{lem2}, we have
\begin{eqnarray}\label{a49}
|\zeta(\xi,z)|+|\nabla_{\xi}\zeta(\xi,z)|\leq CR^{-1}.
\end{eqnarray}
Notice that 
\begin{eqnarray*}
{\rm Im}G(\xi,z)=\left\{\begin{array}{ll}
         \frac{1}{4}J_0(\sigma|\xi-z|)+{\rm Im}G^s(\xi,z)\qquad\qquad&{\rm for}\;\; \xi,z\in\R^2_+\;{\rm or}\;\xi,z\in\R^2_-, \\[3mm]
         {\rm Im}G^s(\xi,z)&{\rm for}\;\;\xi\in\R^2_+, z\in\R^2_-\;{\rm or}\;\xi\in\R^2_-, z\in\R^2_+
       \end{array}
\right.
\end{eqnarray*}
where $J_0$ denotes the Bessel function of order zero, $\sigma=\kappa_1$ for $\xi,z\in\R^2_+$ and $\sigma=\kappa_2$ for $\xi,z\in\R^2_-$. By the smoothness of $G^s(\xi,z)$ and $J_0$, we obtain 
\begin{eqnarray}\label{a50}
\left|{\rm Im} G(\xi,z)\right|+\left|\frac{\partial {\rm Im}G(\xi,z)}{\partial\nu(\xi)}\right|\leq C
\end{eqnarray}
for $\xi\in B\cup D$ and $z\in S$. Since $\psi(\xi,z)$ and $\varphi(\xi,z)$ solve Problem (\ref{a46}) with the boundary data  $(\beta\chi(\xi){\rm Im}G(\xi,z), -{\rm Im}G(\xi,z))$ and $(\beta\chi(\xi)\zeta(\xi,z), -\zeta(\xi,z))$, respectively, a straightforward application of Theorem 3.1 in \cite{LYZ22} leads to
\begin{eqnarray}\nonumber
&&\|\psi(\xi,z)\|_{H^{\frac{1}{2}}(\partial D)}+\left\|\frac{\partial \psi(\xi,z)}{\partial\nu(\xi)}\right\|_{H^{-\frac{1}{2}}(\partial D)}+\|\psi(\xi,z)\|_{L^2(B)}\\\nonumber
&&\lesssim \|{\rm Im}G(\xi,z)\|_{L^2(B)}+\|{\rm Im}G(\xi,z)\|_{H^{\frac{1}{2}}(\partial D)}\\\label{a51}
&&\leq C
\end{eqnarray}
and 
\begin{eqnarray}\nonumber
&&\|\varphi(\xi,z)\|_{H^{\frac{1}{2}}(\partial D)}+\left\|\frac{\partial \varphi(\xi,z)}{\partial\nu(\xi)}\right\|_{H^{-\frac{1}{2}}(\partial D)}+\|\varphi(\xi,z)\|_{L^2(B)}\\\nonumber
&&\lesssim \|\zeta(\xi,z)\|_{L^2(B)}+\|\zeta(\xi,z)\|_{H^{\frac{1}{2}}(\partial D)}\\\label{a52}
&&\leq CR^{-1},
\end{eqnarray}
where we use the trace theorem in the first step of (\ref{a51}) and (\ref{a52}). Thus, with the help of (\ref{a49})-(\ref{a52}), we obtain 
\begin{eqnarray*}
\|\eta(z)\|_{L^{\infty}(S)}\leq CR^{-1}
\end{eqnarray*}
with $C$ depending on $B$ and $D$. The proof is finished.
\end{proof}

\section{Numerical examples}
In this section, we first analyze the behavior of the indicator $\widetilde{\rm Ind}(z)$ when $z$ is close to $B\cup D$, and then present 
several numerical examples to show the effectiveness of the RTM method.

Due to Theorem \ref{thm}, it is found that the behavior of  $\widetilde{\rm Ind}(z)$ depends on the function $\psi(\xi,z)$ when the source and measurement radius $R$ is large enough. Here, the function $\psi(\xi,z)$ solves Problem (\ref{a37}) with boundary data $(\beta\chi(\xi){\rm Im}G(\xi,z), -{\rm Im}G(\xi,z))$. Notice that 
\begin{eqnarray*}
{\rm Im}G(\xi,z)=\left\{\begin{array}{ll}
         \frac{1}{4}J_0(\sigma|\xi-z|)+{\rm Im}G^s(\xi,z)\qquad\qquad&{\rm for}\;\; \xi,z\in\R^2_+\;{\rm or}\;\xi,z\in\R^2_-, \\[3mm]
         {\rm Im}G^s(\xi,z)&{\rm for}\;\;\xi\in\R^2_+, z\in\R^2_-\;{\rm or}\;\xi\in\R^2_-, z\in\R^2_+
       \end{array}
\right.
\end{eqnarray*}
where $\sigma=\kappa_1$ for $\xi,z\in\R^2_+$ and $\sigma=\kappa_2$ for $\xi,z\in\R^2_-$. Observing from (a) of Figure \ref{f2}, the function $J_0(\lambda|\xi-z|)$ achieves a maximum at $\xi=z$, thus, we guess that ${\rm Im}G(\xi,z)$ will also achieves a maximum at $\xi=z$, which is conformed numerically in (b) and (c) of Figure \ref{f2}. Based on this observation, we can expect that $\widetilde{\rm Ind}(z)$ will peak on $B\cup D$.

\begin{figure}[htp]
\begin{center}
\subfigure[$J_0$]{\includegraphics[width=0.31\textwidth]{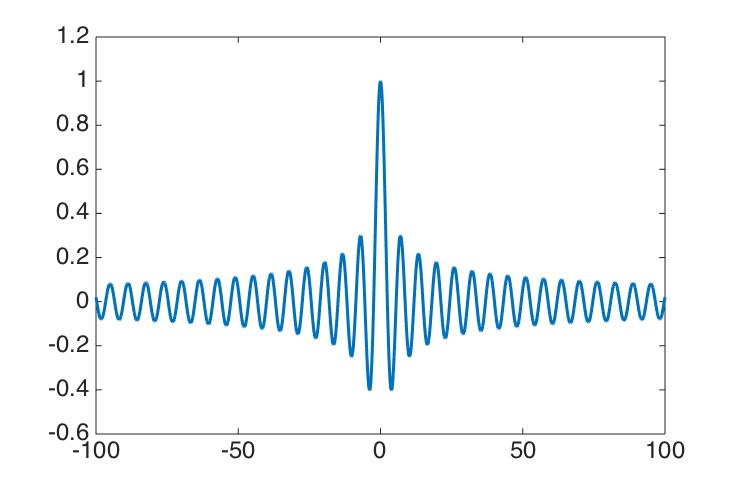}}
\subfigure[${\rm Im}G(x,z)$ with $z=(0,0.5)$ ]{\includegraphics[width=0.31\textwidth]{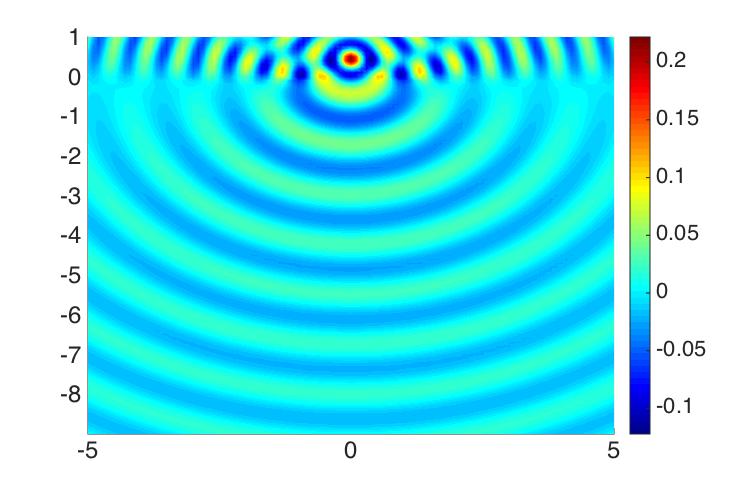}}
\subfigure[${\rm Im}G(x,z)$ with $z=(0,-0.5)$]{\includegraphics[width=0.31\textwidth]{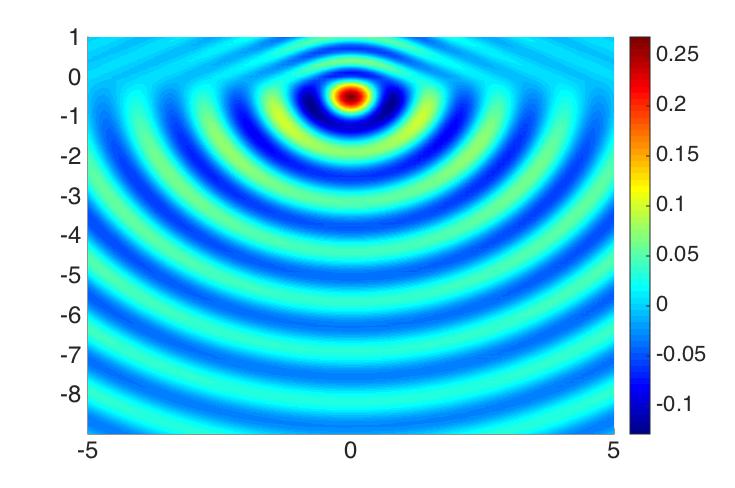}}
\caption{The image of functions $J_0$,  ${\rm Im}G(\xi,z)$  with $\xi\in [-5, 5]\times[-8.95,1.05]$, $\kappa_1=10$, and $\kappa_2=5$.}\label{f2} 
\end{center}
\end{figure}

In all examples, if not stated otherwise, we assume that the wavenumber $\kappa_1=10$, $\kappa_2=5$, the locally rough interface function $f$ is supported in the interval $[-5, 5]$, the sample domain $S=[-5, 5]\times[-8.95, 1.05]$, $N_s=N_r=1024$, and $R=100$. 
The synthetic data $u^s(x_r,x_s)$ and $G^s(x_r,x_s)$ are generated by solving Problems (\ref{a5}) and (\ref{a10}) through the 
Nystr\"{o}m method \cite{LYZ13}. For some relative error $\tau>0$, we inject some noise into the data by defining 
\begin{eqnarray*}
u_{\tau}^s(x_r,x_s)=u^s(x_r,x_s)+\tau\frac{\lambda}{\|\lambda\|_2}\|u^s(x_r,x_s)\|_2
\end{eqnarray*}
where $\lambda=\lambda_1+{\rm i}\lambda_2$ is complex-valued with $\lambda_1$ and $\lambda_2$ consisting of random numbers obeying standard normal distribution $N(0,1)$. 

{\bf Example 1.} In this example, we consider two simple cases and exam the RTM method at different noise level. The first one (see (a) in Figure \ref{f3}) is related to a planar interface $\Gamma=\Gamma_0$ and a circle given by
\begin{eqnarray*}
x(\theta)=(0.5\cos(\theta), -4+0.5\sin(\theta)),\qquad \theta\in[0,2\pi).
\end{eqnarray*}
The second case (see (d) in Figure \ref{f3}) is just related to a locally rough interface $\Gamma$ which is described by
\begin{eqnarray*}
f(t)=\Omega_3(2t+4)-0.6\Omega_3(2t-5)
\end{eqnarray*}
with $\Omega_3(\cdot)\in C_0^2(\mathbb R)$ being a cubic $B$-spline function given by
\ben
\Omega_3(t)
= \left\{\begin{aligned}
\frac{1}{2}|t|^3-t^2+\frac{2}{3}     &\qquad\qquad \textrm{for}\; |t|\leq1,\\
-\frac{1}{6}|t|^3+t^2-2|t|+\frac{4}{3} & \qquad\qquad \textrm{for}\;1<|t|<2\,,\\
 0 &\qquad\qquad\textrm{for}\; |t|\geq2.
\end{aligned}
\right.
\enn
The reconstruction results are presented in Figure \ref{f3}, which shows that the RTM method can provide a satisfied reconstruction for these two simple cases, even for $10\%$ noisy data.

{\bf Example 2.} In this example, we test the dependence of the RTM algorithm on the relative position and distance between the local 
perturbation $B$ and the embedded obstacle $D$. The locally rough interface $\Gamma$ and the rounded square shaped obstacle $D$ are parameterized by
\ben
&&f(t)=\left(0.6e^{-6(t+3)^2}+0.5e^{-7t^2}+0.5e^{-8(t-3)^2}\right)\cdot f_0(t)\quad t\in \R,\\
&&x(\theta)=(3+0.3(\cos^3(\theta)+\cos(\theta)),-6+0.3(\sin^3(\theta)+\sin(\theta))),\; \theta\in [0,2\pi),
\enn
where $f_0(t)\in C_0^{\infty}(\R)$ is a cut-off function defined by
\ben
f_0(t)
= \left\{\begin{aligned}
1 &\qquad\qquad \textrm{for}\; |t|\leq4,\\
\;\left(1+e^{\frac{1}{5-|t|}+\frac{1}{4-|t|}}\right)^{-1} & \qquad\qquad \textrm{for}\;4<|t|<5\,,\\
 0 &\qquad\qquad\textrm{for}\; |t|\geq5;
\end{aligned}
\right.
\enn
see (a) in Figure \ref{f4}. Next, we fix the locally rough interface $\Gamma$ and move the embedded obstacle $D$ up by 
four units and five units, see (b) and (c) in Figure \ref{f4}. The numerical results from $10\%$ noisy data are illustrated in 
(d), (e), (f) of Figure \ref{f4}. It is readily seen from Figure \ref{f4} that the quality of the reconstruction will depend on the 
relative position of $B$ and $D$. We guess that a possible reason is due to the strong multiple scattering between $B$
and $D$ when $D$ is close to $\Gamma$.

{\bf Example 3.} In the last example, as an attempt, we consider a piecewise continuous locally rough interface and exam
the RTM algorithm at different radiuses $R$. The locally rough interface is described by a piecewise constant defined by 
\begin{equation}\nonumber
  f(t)=\left\{\begin{array}{lll}
                         0.2 \qquad\;\; |t|\leq 1, \\[1mm]
                         0.3 \qquad\;\; 3\leq|t|\leq 4,\\[1mm]
                         0  \qquad\;\;{\rm others}.
                          \end{array}\right.
\end{equation}
The embedded obstacle is rounded triangle shaped given by 
\ben
x(\theta)=(-3+\left(0.5+0.1\cos(3\theta)\right)\cos(\theta),-6+\left(0.5+0.1\cos(3\theta)\right)\sin(\theta)),\; \theta\in [0,2\pi),
\enn
see (a) in Figure \ref{f5}. The radius is set to be $R=20$ and $R=100$, and the noise level is $10\%$. 
We present the reconstructions in (b), (c) of Figure \ref{f5}, which shows that the RTM algorithm is able to provide high reconstruction 
quality at these measurement radiuses.

It is observed from Figure \ref{f3}, Figure \ref{f4}, and Figure \ref{f5} that the RTM approach proposed in Theorem \ref{thm} can provide
accurate and stable reconstructions of the locally rough interface as well as the embedded obstacles for a variety of interfaces and obstacles. The image quality depends on the interaction of the locally rough interface and the embedded obstacle. Moreover, as shown in Figure \ref{f3}, Figure \ref{f4}, and Figure \ref{f5}, the RTM method is robust to noise.

\begin{figure}[htp]
\begin{center}
\subfigure[Physical configuration]{\includegraphics[width=0.31\textwidth]{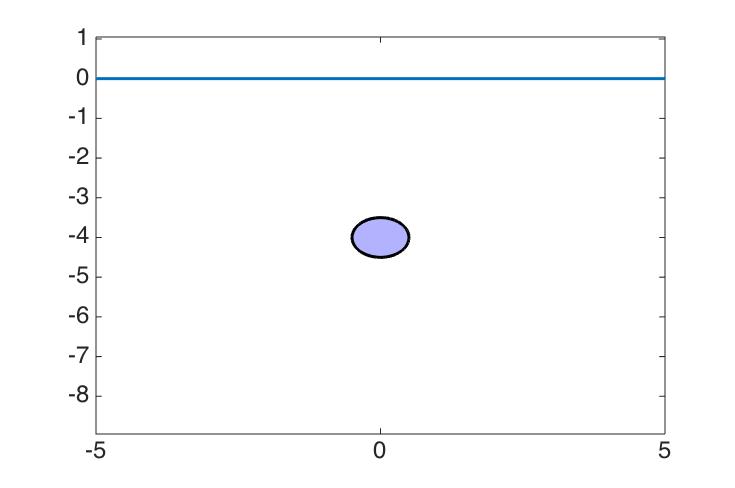}}
\subfigure[No noise]{\includegraphics[width=0.31\textwidth]{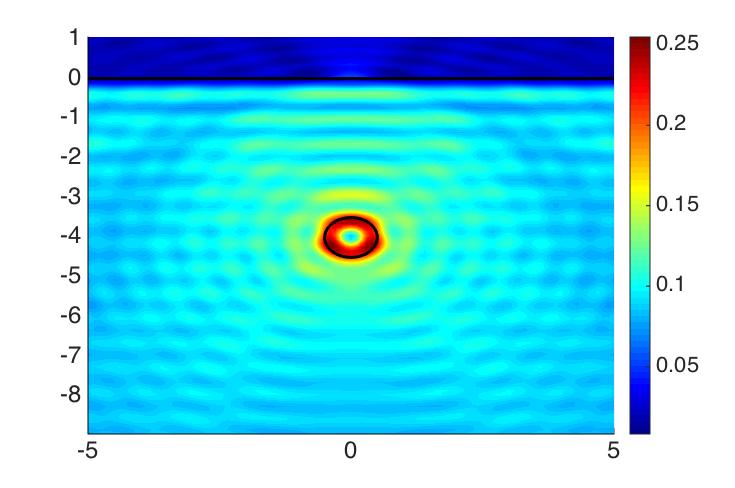}}
\subfigure[$10\%$ noise]{\includegraphics[width=0.31\textwidth]{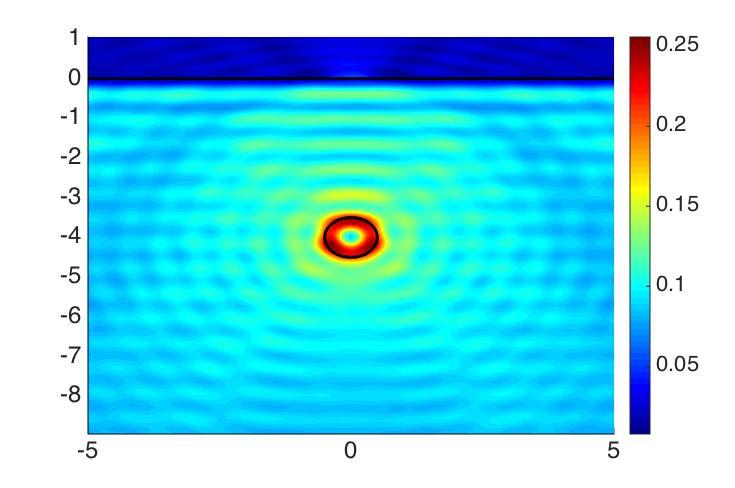}}
\subfigure[Physical configuration]{\includegraphics[width=0.31\textwidth]{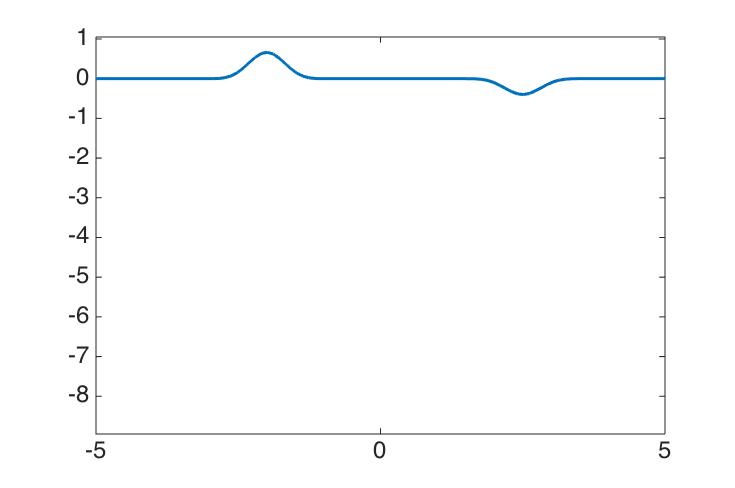}}
\subfigure[No noise]{\includegraphics[width=0.31\textwidth]{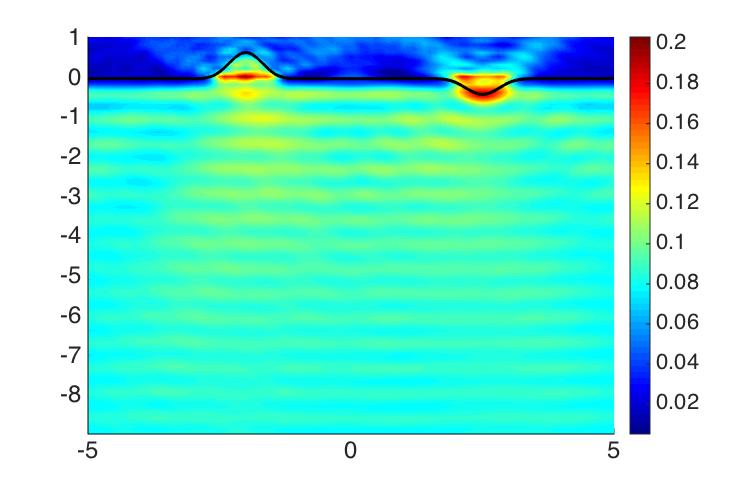}}
\subfigure[$10\%$ noise]{\includegraphics[width=0.31\textwidth]{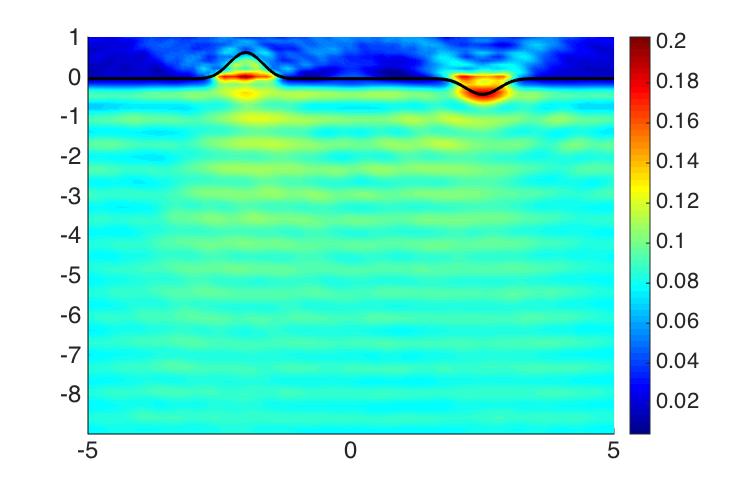}}
\caption{Reconstructions of the two simple cases in Example 1 from data at different noise level.}\label{f3} 
\end{center}
\end{figure}

\begin{figure}[htp]
\begin{center}
\subfigure[Physical configuration]{\includegraphics[width=0.31\textwidth]{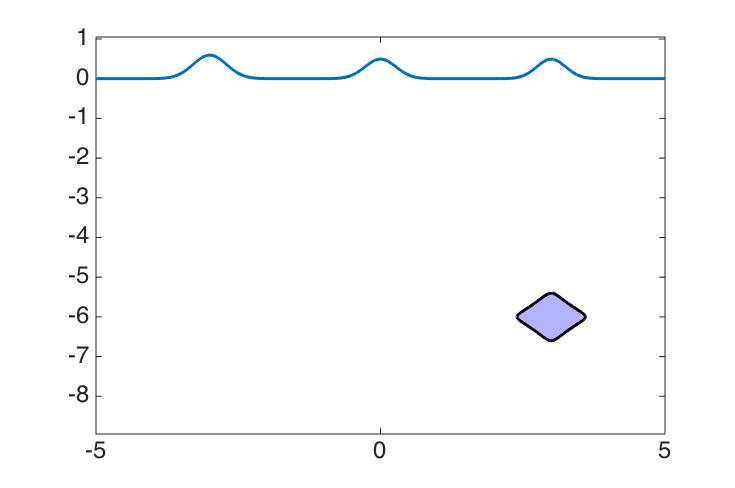}}
\subfigure[Physical configuration]{\includegraphics[width=0.31\textwidth]{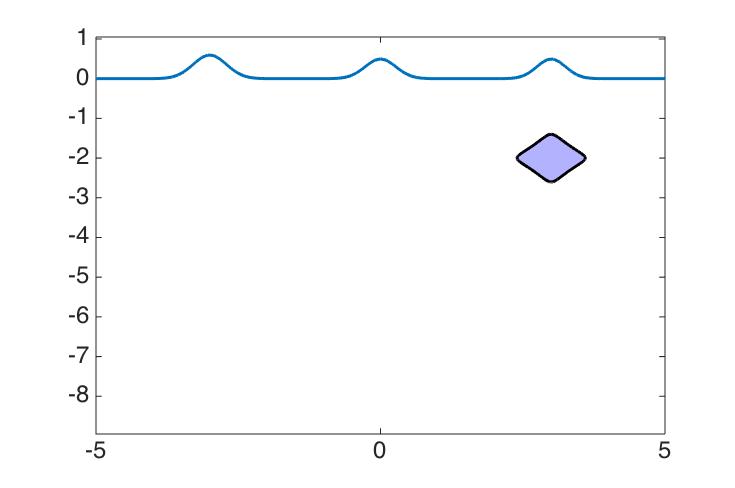}}
\subfigure[Physical configuration]{\includegraphics[width=0.31\textwidth]{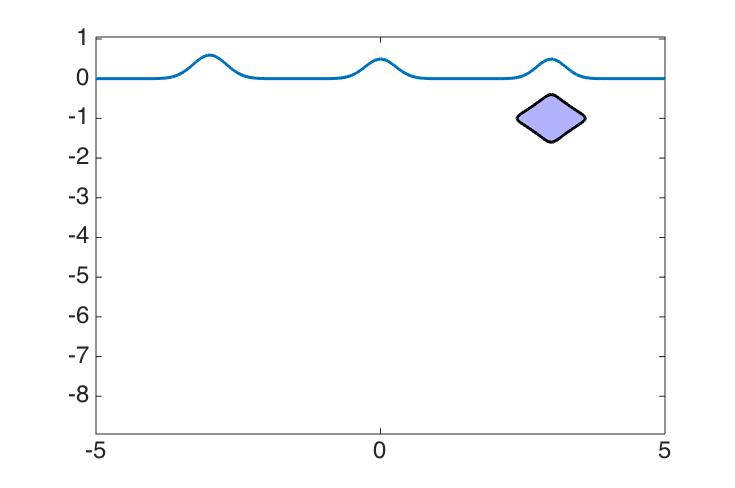}}
\subfigure[$10\%$ noise]{\includegraphics[width=0.31\textwidth]{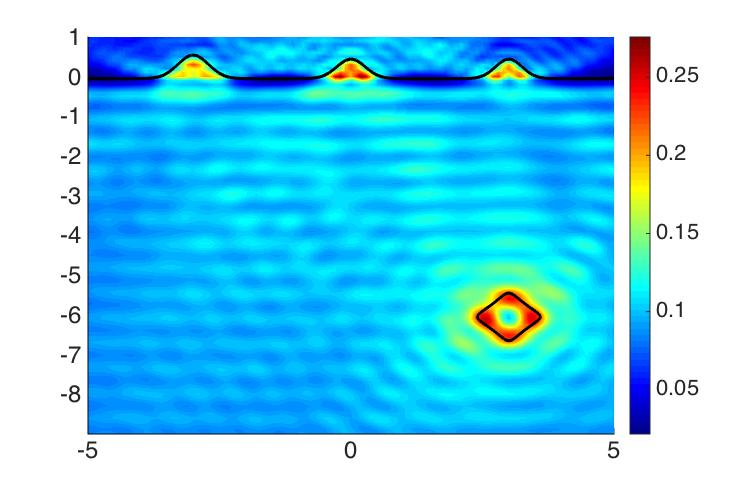}}
\subfigure[$10\%$ noise]{\includegraphics[width=0.31\textwidth]{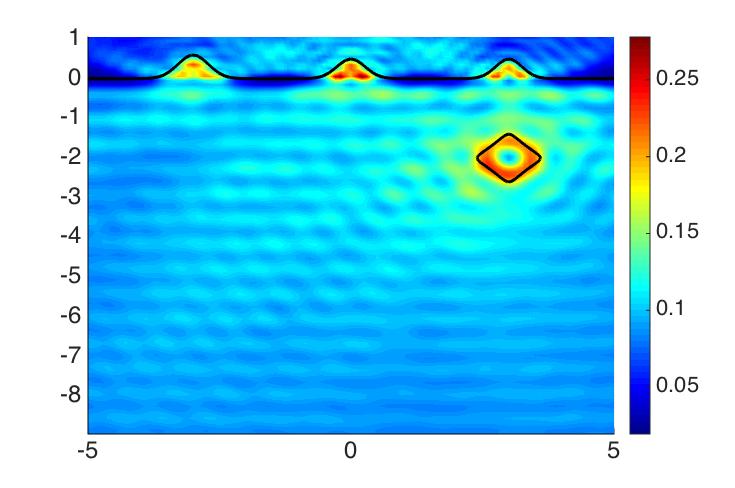}}
\subfigure[$10\%$ noise]{\includegraphics[width=0.31\textwidth]{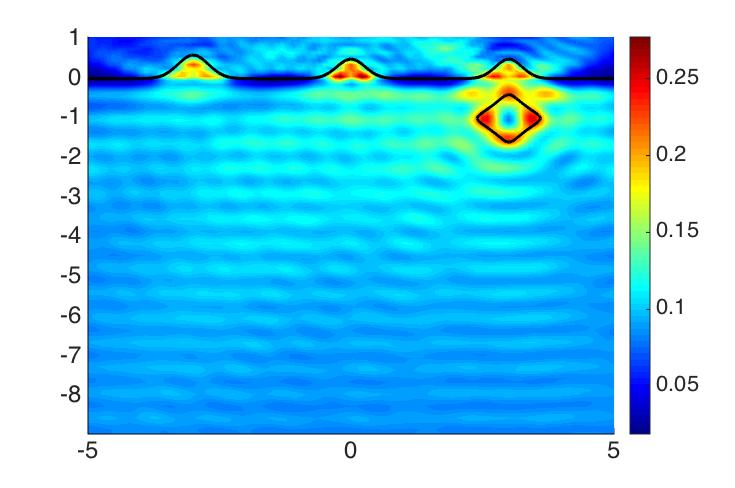}}
\caption{Reconstructions results for different relative position between the locally rough interface and the embedded obstacle.}\label{f4} 
\end{center}
\end{figure}

\begin{figure}[htp]
\begin{center}
\subfigure[Physical configuration]{\includegraphics[width=0.31\textwidth]{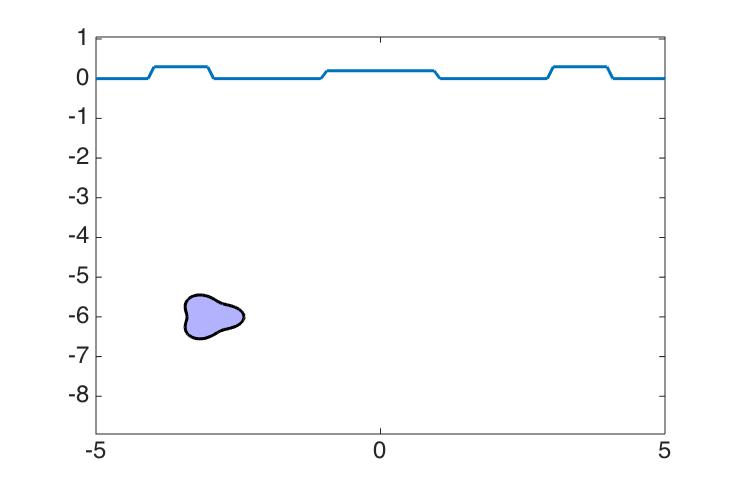}}
\subfigure[$R=20$, $10\%$ noise]{\includegraphics[width=0.31\textwidth]{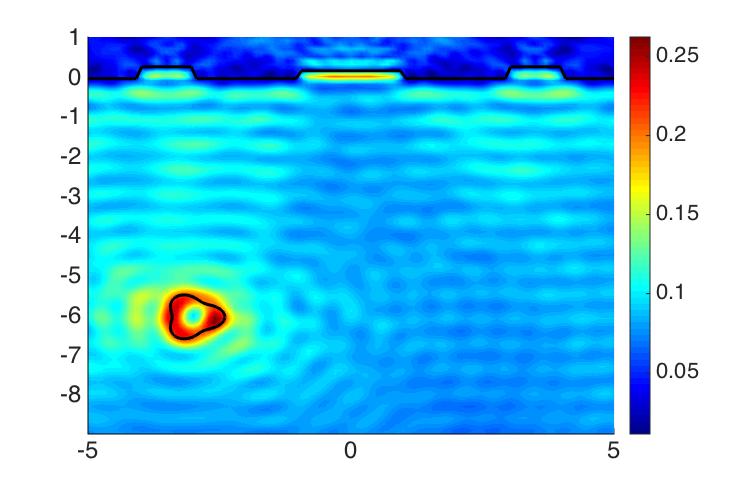}}
\subfigure[$R=100$, $10\%$ noise]{\includegraphics[width=0.31\textwidth]{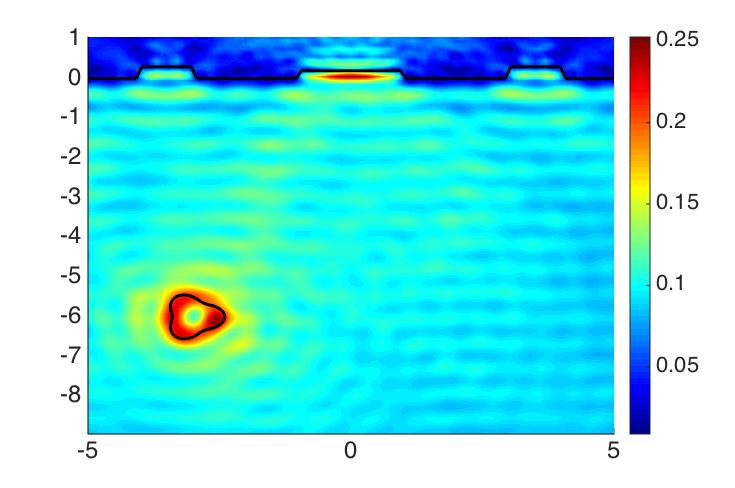}}
\caption{Reconstructions of the locally rough interface and the object given in Example 3 for different measurement radiuses.}\label{f5} 
\end{center}
\end{figure}

\section{Conclusion}
In this paper, we proposed an extended RTM method to simultaneously recover a locally rough interface and an embedded object 
in the lower half-space from near-field measurements. The main idea is based on constructing a modified Helmholtz-Kirchhoff identity
associated with the planar interface. Numerical experiments showed that the novel RTM algorithm can provide an accurate and 
stable reconstruction for a large number of locally rough interfaces and embedded obstacles, even for piecewise continuous interfaces.
Furthermore, the reconstructions can be regarded as a good initial guess for an iterative algorithm to obtain a more accurate result. 
As far as we know, this is the first RTM method to simultaneously recover an unbounded scatterer and a bounded scatterer. It is easily seen that the RTM method proposed in Theorem \ref{thm} depends crucially on the priori information that the interface is locally perturbed. However, if the rough surface is non-local, it is unclear how to establish the related Helmholtz-Kirchhoff identity and then 
extend this method to reconstruct the non-local rough surface. We hope to report the progress on this topic in the future.

{\bf Acknowledgments.} This work was supported by the NNSF of China grants No. 12171057, 12122114, and Education
Department of Hunan Province No. 21B0299.

\end{document}